\documentclass[12pt,reqno]{amsart}
\usepackage{amssymb,mathrsfs}
\usepackage{colortbl}
\usepackage[]{xcolor}
\numberwithin{equation}{section}
\newtheorem{theorem}{Theorem}[section]
\newtheorem{corollary}[theorem]{Corollary}
\newtheorem{lemma}[theorem]{Lemma}
\newtheorem{proposition}[theorem]{Proposition}
\newtheorem{remark}[theorem]{Remark}
\begin{document}
\title[NLS with double nonlinear terms]{Scattering for the mass super-critical \\perturbations of the mass critical nonlinear Schr\"odinger equations}

\author[Xing Cheng]{Xing Cheng}
\address{College of Science, Hohai University,\\
 Nanjing 210098, Jiangsu, China}
\email{chengx@hhu.edu.cn}
\subjclass{Primary 35Q55; Secondary 35L70}
\keywords{Nonlinear Schr\"odinger equation; double nonlinearities, profile decomposition; scattering, blowup.}
\begin{abstract}
We consider the Cauchy problem for the nonlinear
Schr\"odinger equation with double nonlinearities with opposite sign, with one term is mass-critical and the other term is mass-supercritical and energy-subcritical, which includes the famous two-dimensional cubic-quintic nonlinear Schr\"odinger equaton. We prove global wellposedness and scattering in $H^1(\mathbb{R}^d)$ below the threshold for non-radial data when $1 \le d \le 4$.
\end{abstract}

\maketitle

\section{Introduction}
In this article, we will mainly consider the following two kinds of nonlinear Schr\"odinger equations with double nonlinearities:
\begin{equation}\label{eq1.1}
\begin{cases}
i\partial_t u  + \Delta u = |u|^\frac4d u - |u|^{p-1} u,\\
u(0) = u_0 \in H^1(\mathbb{R}^d),
\end{cases}
\end{equation}
and
\begin{equation}\label{eq1}
\begin{cases}
i\partial_t u  + \Delta u = -|u|^\frac4d u + |u|^{p-1} u,\\
u(0) = u_0 \in H^1(\mathbb{R}^d),
\end{cases}
\end{equation}
where $u: \mathbb{R}  \times \mathbb{R}^d \to \mathbb{C}$ is a complex-valued function,
$1+\frac{4}{d} < p  <  1+\frac{4}{d-2},\  d =  3, 4, $ and $1 + \frac4d < p < \infty,\  d = 1,\,2$.

Equations \eqref{eq1.1} and \eqref{eq1} are special cases of the nonlinear Schr\"odinger equations with double nonlinearities
\begin{equation}\label{gmc}
\begin{cases}
i\partial_t u  + \Delta u = \mu_1|u|^{p_1-1} u + \mu_2|u|^{p_2-1} u,\\
u(0) = u_0 \in H^1(\mathbb{R}^d),
\end{cases}
 \end{equation}
 where $1 < p_1 < p_2\leq 1+\frac{4}{d-2}$, for $d\ge 3$, $\ 1< p_1 <  p_2 < \infty$, for $d = 1,\,2$, $\mu_1,\mu_2\in\{\pm1\}$. For the physical background of these kind of equations, we refer the reader to \cite{Barashenkov-Gocheva-Makhankov-Puzynin,F,Pelinovsky-Afanasjev-Kivshar}, and also the introduction in \cite{CMZ}. These kind of equations have been studied intensively in the past decade. In \cite{Akahori-Ibrahim-Kikuchi-Nawa1,Akahori-Ibrahim-Kikuchi-Nawa2,CMZ,KTPV,Miao-Xu-Zhao2,MZZ,Miao-Xu-Zhao1,Tao-Visan-Zhang,Xie,XY,X-Zhang}, the authors study the well-posedness and scattering of these kind of equations.

The local wellposedness theory for \eqref{gmc} can be given by using Banach's fixed point theorem as in \cite{Cazenave,Cazenave-Weissler}. X. Zhang \cite{X-Zhang} investigated the wellposedness, scattering and blowup of \eqref{gmc} in three dimensional case where $p_2 = 5$, where she viewed the equation as a perturbation of the energy-critical defocusing nonlinear Schr\"odinger equation, which has been proven global well-posedness and scatters by J. Colliander, M. Keel, G. Staffilani, H. Takaoka, and T. Tao \cite{CKSTT}.
  T. Tao, M. Visan and X. Zhang considered various cases in \cite{Tao-Visan-Zhang} when $d \ge 3$. In particular, they proved global wellposedness and scattering of the solution to the equation (\ref{gmc}) when $\mu_1=\mu_2= 1$ and $1 + \frac{4}{d}\leq p_1<p_2  \le 1 +  \frac{4}{d-2},\  d \ge 3$.

When $\mu_1=1$, $\mu_2=-1$, C. Miao, G. Xu and L. Zhao \cite{Miao-Xu-Zhao1} considered the case where $p_1 = 3, p_2 = 5$, $d = 3$. The threshold was given by variational method as in \cite{Ibrahim-Masmoudi-Nakanishi} and they also established the linear profile decomposition in $H^1$ in the spirt of \cite{Ibrahim-Masmoudi-Nakanishi}. By using the profile decomposition, they reduced the scattering problem to the extinction of the critical element. The critical element can then be excluded by using the virial identity. They showed the dichotomy of global wellposedness and scattering versus blow-up below the threshold for radial solutions. We also refer to \cite{XY} for further discussion about $1 + \frac{4}{d} < p_1 < p_2 = 1+ \frac{4}{d-2}$, $d = 3$. The radial assumption was removed in dimensions four and higher in \cite{Miao-Xu-Zhao2,MZZ} by using \cite{D5,Killip-Visan}. For the case $1 + \frac{4}{d} = p_1< p_2 \le 1 + \frac{4}{d-2}$, $1\le d\le 4$, X. Cheng, C. Miao, and L. Zhao \cite{CMZ} proved global wellposedness and scattering in $H^1$ in the radial case where they gave the exact value of the threshold which is related to the ground state. We also refer to the recent work of N. Soave \cite{Soave1,Soave} on the study of the properties of ground states of the equations in this case.

When $\mu_1=\mu_2=-1$, T. Akahori, S. Ibrahim, H. Kikuchi, and H. Nawa \cite{Akahori-Ibrahim-Kikuchi-Nawa1, Akahori-Ibrahim-Kikuchi-Nawa2} considered \eqref{gmc} when $1 + \frac{4}{d}< p_1 < p_2 = 1 + \frac{4}{d-2}$, $d \ge 5$. After giving existence of the ground state based on the idea in \cite{Brezis-Nirenberg} and \cite{Ibrahim-Masmoudi-Nakanishi}, they showed a sufficient condition for the global wellposedness and scattering by using the scattering result of the energy-critical focusing nonlinear Schr\"odinger equation in \cite{Killip-Visan}. They also obtained the nine-set theory developed by K. Nakanishi and W. Schlag \cite{NS,NS1} when the Lyapunov functional slightly above the ground state threshold at low frequencies in \cite{AIKN3}. J. Xie \cite{Xie} proved global wellposedness and scattering below the threshold for $1 + \frac{4}{d}< p_1 < p_2 <  1 + \frac{4}{d-2}$ when $d \ge 3$, $ 1 + \frac{4}{d} < p_1 < p_2  <  \infty$ when $ d = 1,2$, by using the argument in \cite{Duyckaerts-Holmer-Roudenko,Holmer-Roudenko}.

If $\mu_1 = - 1$, $\mu_2 = 1$, with $p_1 = 3$, $p_2 = 5$ when $d = 3$, which is the three dimensional cubic-quintic nonlinear Schr\"odinger equation, R. Killip, T. Oh, O. Pocovnicu, and M. Visan \cite{KTPV} analyzed the one-parameter family of ground-state solitons associated to the equation, and proved scattering when the solutions belong to some region of the mass/energy plane where the virial is positive.

In the meantime, there are a lot of work on the study of the stability and instability of the solitary wave of \eqref{gmc}, we refer to \cite{FO,Ma,Oh,OY} and the references therein and also the recent work of A.~Stefanov \cite{St}. We refer to \cite{GLT,Le Coz-Tsai} for the construction of solutions build upon solitons and kinks.

We also refer to the recent work by V. D. Dinh and B. Feng \cite{DF} and the references therein on the fractional nonlinear Schr\"odinger equation with double nonlinearities.

We will first consider the equation \eqref{eq1.1}, it has the conserved quantity: mass, energy, and momentum, defined respectively to be
\begin{align*}
&   \mathcal{M}(u)  = \int_{\mathbb{R}^d} |u(t,x )|^2\,\mathrm{d}x, 
\\
& \mathcal{E}(u)  = \int_{\mathbb{R}^d} \Big(\frac12 |\nabla u(t,x )|^2 - \frac1{p+1} |u(t,x ) |^{p+1} + \frac{d}{2(d+2)} |u(t,x )|^{\frac{2(d+2)}d}\Big) \,\mathrm{d}x, 
\\
& \mathcal{P}(u )  = \mathrm{\Im} \int_{\mathbb{R}^{d}}\nabla u(t,x )\overline{u(t,x)}\,dx. 
\end{align*}
 For $\varphi \in H^1(\mathbb{R}^d)$ and any $ \omega > 0$, we have the Lyapunov functional
\begin{equation*}
\mathcal{S}_\omega (\varphi) = \mathcal{E}(\varphi) + \frac12\omega \mathcal{M}(\varphi),
\end{equation*}
which has the scaling derivative $\mathcal{K}(\varphi)$ defined to be
\begin{equation*}
\begin{split}
\mathcal{K}(\varphi) = \int_{\mathbb{R}^d} |\nabla \varphi|^2 - \frac{d(p-1)}{2(p+1)} |\varphi|^{p+1} + \frac{d}{d+2} |\varphi|^\frac{2(d+2)}d \,\mathrm{d}x.
\end{split}
\end{equation*}
Let
\begin{equation*}
  m_\omega = \inf\left\{ \mathcal{S}_\omega(\varphi): \varphi \in H^1(\mathbb{R}^d)\setminus \{0\}, \mathcal{K} (\varphi) = 0 \right\},
\end{equation*}
by \cite{CMZ}, we have $m_\omega = \mathcal{S}_\omega(Q) > 0$, where $Q\in H^1(\mathbb{R}^d)$ is the ground state of
 \begin{equation} \label{eq2.8}
- \omega Q + \Delta Q + |Q|^{p-1} Q - |Q|^\frac4d Q = 0.
\end{equation}
We now state our result concerning the scattering versus blow-up dichotomy below the ground state energy for \eqref{eq1.1} in the non-radial case.
\begin{theorem}\label{th1.4}
For any $\omega > 0$, \\
 $(i)$ if $u_0 \in A_{\omega,+}$, the solution $u$ to $\eqref{eq1.1}$ exists globally and scatters in $H^1(\mathbb{R}^d)$;\\
 $(ii)$ if $u_0 \in A_{\omega,-}$, the solution $u$ blows up in finite time in the non-radial case if $u_0$ is in the weighed space $\Sigma$ or in the radial case for $d\ge 2, \ p \le \min(5, 1+ \frac4{d-2})$, where
\begin{align*}
& A_{\omega, +} = \left\{ \varphi\in H^1(\mathbb{R}^d): \mathcal{S}_\omega(\varphi) < m_\omega,\, \mathcal{K}(\varphi) \ge 0 \right\},\\
& A_{\omega, -} = \left\{ \varphi\in H^1(\mathbb{R}^d): \mathcal{S}_\omega(\varphi) < m_\omega,\, \mathcal{K}(\varphi) < 0\right\},
\end{align*}
\end{theorem}
We now turn to the equation \eqref{eq1}, it turns out to be the famous two-dimensional cubic-quintic nonlinear Schr\"odinger equation \cite{F} when $d =2$ and $p = 5$. It is known that \eqref{eq1} is global wellposedness and scatters for small data, see \cite{Tao-Visan-Zhang}. However, we do not know the long time behavior of the solution for large data. Now by using the sharp Gagliardo-Nirenberg inequality in our proof, we can give the large data scattering result of \eqref{eq1} by using the argument of Theorem \ref{th1.4}, and also use the scattering result of the mass-critical focusing nonlinear Schr\"odinger equation in \cite{D4}. We now give our result on the scattering for \eqref{eq1} when the mass is less than the mass of the ground state solution of \eqref{eq1.10v30}.
\begin{theorem}\label{th1.3v16}
If $u_0\in H^1(\mathbb{R}^d)$, and $\left\|u_0\right\|_{L^2(\mathbb{R}^d)} < \left\|Q\right\|_{L^2(\mathbb{R}^d)}$, the solution $u$ to \eqref{eq1} exists globally and scatters in $H^1(\mathbb{R}^d)$, where $Q$ is the ground state of
\begin{align}\label{eq1.10v30}
- Q + \Delta Q + |Q|^\frac4d Q = 0.
\end{align}
\end{theorem}
\begin{remark}
$\left\|Q\right\|_{L^2(\mathbb{R}^d)}$ is the exact threshold for scattering, see \cite{Le Coz-Tsai}, solution with the soliton as the main part can be constructed, which does not scatter can occur for \eqref{eq1} if $\left\|u_0\right\|_{L^2(\mathbb{R}^d)} >  \left\|Q\right\|_{L^2(\mathbb{R}^d)}$.
\end{remark}
\begin{remark}
 Most of the argument in the proof of this theorem is similar to the proof of Theorem \ref{th1.4}, we will give a sketch of the proof in Section \ref{se8v19}.
\end{remark}
We describe now some of the ideas involved in the proof, and mainly the proof of Theorem \ref{th1.4}. To deal with the nonlinear Schr\"odinger equation with double nonlinearities, we focus on the different roles played by the two nonlinearities, it is a little difficult to figure out which term in the nonlinearities dominated the long time behaviour of the solution. Generally speaking, the focusing term will greatly affect the behaviour of the solution, since it can cause finite time blow-up, but on the other hand the defocusing term can preclude the blow-up in some sense. By the variational computation made in \cite{CMZ}, we can see under the assumption that the Lyapunov functional of the initial data is less than the Lyapunov functional of the ground state solution, we can give a global wellposedness versus blow-up dichotomy by using the scaling functional $\mathcal{K}$. This reveals the interesting fact that under suitable condition on the initial data, the focusing term will dominated the defocusing term. Therefore, the remaining difficult part is to prove the scattering when the solution globally exists, we prove the scattering result by the compactness-contradiction method initiated by C. E. Kenig and F. Merle \cite{Kenig-Merle,KM2}. To prove scattering, it is a little involved because the defocusing term is mass-critical, and we need to show the solution has a finite space-time norm $L_t^\frac{2(d+2)}d W_x^{1, \frac{2(d+2)}d}(\mathbb{R}\times \mathbb{R}^d)$, but since our nonlinearities are energy-subcritical, we find a weaker finite space-time norm $L_t^\frac{2(d+2)}d W_x^{s_p, \frac{2(d+2)}d}(\mathbb{R}\times \mathbb{R}^d)$, where $s_p = \frac{d}2- \frac2{p-1} \in (0, 1)$, is enough to yield scattering. As a result, we only need to establish a linear profile decomposition in $H^1$ with the remainder asymptotically vanishes in $L_t^\frac{2(d+2)}d W_{x}^{s_p, \frac{2(d+2)}d}(\mathbb{R}\times \mathbb{R}^d)$, which is equivalent to describe the lack of compactness of the embedding $e^{it\Delta}: H^1_x( \mathbb{R}^d) \hookrightarrow L_t^\frac{2(d+2)}d W_{x}^{s_p, \frac{2(d+2)}d}(\mathbb{R}\times \mathbb{R}^d) $. Then as in \cite{CGZ}, we can directly apply the linear profile decomposition of the Schr\"odinger equations in $L^2$, but keep our eyes open that the initial data is in $H^1$, we can exclude one direction of the scaling limit.
After making the contradiction that the solution does not scatter, we can find a sequence of solutions $u_n: \mathbb{R} \times \mathbb{R}^d \to \mathbb{C}$ with
$u_n \in A_{\omega, +}$, and
\begin{align*}
&  \ \mathcal{S}_\omega(u_n) \to m_\omega^*, \text{ as } n\to \infty,\\
& \lim\limits_{n\to \infty} \left\| \langle \nabla \rangle^{s_p} u_n \right\|_{L_{t,x}^\frac{2(d+2)}d([0, \infty) \times \mathbb{R}^d)}
= \lim\limits_{n\to \infty} \left\| \langle \nabla \rangle^{s_p} u_n \right\|_{L_{t,x}^\frac{2(d+2)}d((- \infty, 0] \times \mathbb{R}^d)} = \infty,
\end{align*}
where $m_\omega^*$ to be the smallest number for which such a sequence exists, we can applying the linear profile decomposition to the minimizing sequence $u_n(0,x)$, then we need the analyze the nonlinear profiles related to the large-scale case, that is the solution of
\begin{align*}
\begin{cases}
i\partial_t u_n + \Delta u_n = |u_n|^\frac4d u_n - |u_n|^{p-1} u_n, \\
u_n(0,x) = e^{i\theta_n } e^{ix\cdot\xi_n } e^{-it_n \Delta} \left(\frac1{(h_n)^\frac{d}2 } \left( P_{\le h_n^\theta} \varphi\right)  \left(\frac{\cdot - x_n }{h_n}\right)\right)(x),
\end{cases}
\end{align*}
 where $\varphi \in L^2(\mathbb{R}^d)$, $( \theta_n, h_n, t_n, x_n, \xi_n)  \subseteq  \mathbb{R}/2\pi\mathbb{Z} \times (0,\infty)\times  \mathbb{R} \times \mathbb{R}^d \times \mathbb{R}^d$, $| \xi_n | \le C $, $h_n \to \infty$, as $n\to \infty$, and $0 < \theta < 1$, we find that we can approximate this kind of nonlinear profile by using a transformation of the solution of the mass-critical nonlinear Schr\"odinger equation:
\begin{align*}
\begin{cases}
i\partial_t v + \Delta v = |v|^\frac4d v,\\
v(0,x) = v_0(x),
\end{cases}
\end{align*}
where $v_0\in H^1$ close to $\varphi$ in $L^2$, the solution $v$ has finite scattering norm by the scattering theorem of the mass-critical nonlinear Schr\"odinger equation, we can get the solution $u_n$ has finite scattering norm. Finally, we obtain a critical element $u_c \in  C(\mathbb{R}, H^1(\mathbb{R}^d))$, which is almost periodic modulo spatial translation. The spatial translation parameter $x(t)$ obeys $x(t) = o(t)$, as $t \to \pm \infty$ as a consequence of the zero momentum property of the critical element as \cite{Duyckaerts-Holmer-Roudenko,KM2,KTPV}. Then by using the localized virial quantity, we can exclude the critical element, which then reveals the scattering result. On the other hand, for the proof of Theorem \ref{th1.3v16}, most of the argument in the proof is same as in the proof of Theorem \ref{th1.4}, and the main difference is that we use the sharp Gagliardo-Nirenberg inequality instead of the variational estimates from \cite{CMZ}.

We expect our result can be extended to higher dimensions $d\ge 5$, since the only obstacle is the stability theory. However, it seems difficult to control the mass-critical term in the Sobolev spaces $H^1$ or the weaker space $H^{s_p}$ even by using exotic Strichartz estimates in \cite{Foschi, Vilela} when we are trying to establish the stability theory.

The rest of the paper is organized as follows. After introducing some notations and preliminaries, we recall the variational estimate related to \eqref{eq2.8} from \cite{CMZ} in Section \ref{se2v9} and we also give the proof of the blow-up part of Theorem \ref{th1.4} in the non-radial case. The local wellposedness and stability theory are stated in Section \ref{se4v24}. In Section \ref{se5v24}, we derive the linear profile decomposition for data in $H^1(\mathbb{R}^d)$. Then we argue by contradiction. We reduce to the existence of a critical element in Section \ref{se6v24} and show the extinction of such a critical element in Section \ref{se5}. We give a sketch of the proof of Theorem \ref{th1.3v16} in Section \ref{se8v19}.
\subsection{Notation and Preliminaries}
We will use the notation $X\lesssim Y$ whenever there exists some positive constant $C$ so that $X\le C Y$. Similarly, we will use $X\sim Y$ if $X\lesssim Y \lesssim X$.

We define the Fourier transform on $\mathbb{R}^d$ to be
\begin{align*}
\hat{f}(\xi) =\frac1{(2\pi)^d} \int_{\mathbb{R}^d} e^{- ix\xi} f(x)\,\mathrm{d}x,
\end{align*}
 and for $s \in \mathbb{R}$,
the fractional differential operators $|\nabla|^s$ is defined by $\widehat{|\nabla|^s f}(\xi)  = |\xi|^s \hat{f}(\xi).$ We also define
 $\langle \nabla\rangle^s$ by $\widehat{\langle \nabla\rangle^s f}(\xi)=  (  1 + |\xi|^s )  \hat{f}(\xi)$.

We define the homogeneous and inhomogeneous Sobolev norms by
 \[ \left\|f\right\|_{\dot{H}^s(\mathbb{R}^d)} = \left\| |\nabla|^s f\right\|_{L^2(\mathbb{R}^d)}, \
\left\|f\right\|_{H^s(\mathbb{R}^d)} = \left\|\langle \nabla\rangle^s f\right\|_{L^2(\mathbb{R}^d)}.\]

For $I \subseteq \mathbb{R}$, we use $L_t^q L_x^r(I \times \mathbb{R}^d)$ to denote the spacetime norm
$$\left\|u\right\|_{L_t^q L_x^r(I \times \mathbb{R}^d)} =\bigg(\int_I \Big(\int_{\mathbb{R}^d} |u(t,x)|^r \,\mathrm{d}x\Big)^\frac{q}{r}\,\mathrm{d}t\bigg)^\frac1 q.$$
When $q = r$, we abbreviate $L_t^q L_x^r$ as $L_{t,x}^q$.

 We say that a pair of exponents $(q,r)$ is $L^2$-\emph{admissible} if $\frac{2}{q} +
\tfrac{d}{r} = \frac{d}{2}$ and $2 \leq q,r \leq \infty, \ (q,r,d) \ne (2, \infty, 2), \ d\ge 1$.
If $I \times \mathbb{R}^d$ is a space-time slab, for any $ 0 \le s \le 1$, we define the \emph{Strichartz norm} ${S}^s(I)$ by
\begin{equation*}
\|u\|_{{S}^s(I) } = \sup \|\langle \nabla \rangle^{s} u \|_{L_t^q L_x^r(I\times \mathbb{R}^d)},
\end{equation*}
where the $\sup$ is taken over all $L^2$-admissible pairs $(q,r)$ with $ q < \infty$.
When $d = 2$, we need to modify the norm a little, where the $\sup$ is taken over all $L^2$-admissible pairs with $q \ge 2 + \epsilon$,
 for $\epsilon > 0$ arbitrary small.

\section{Variational estimates}\label{se2v9}
In this section, we present the facts related to the ground state in \eqref{eq2.8}, which are proven in \cite{CMZ}. We also give the energy-trapping property for $A_{\omega, \pm}$.
In the end of this section, we give the proof of the blow-up part of Theorem \ref{th1.4} by using the energy-trapping for $A_{\omega, -}$.

Due to the lack of positivity of $\mathcal{S}_\omega(\varphi)$, we introduce a positive functional
\begin{equation}\label{eq2.3new}
\begin{split}
\mathcal{H}_\omega(\varphi) := \mathcal{S}_\omega(\varphi) - \frac12 \mathcal{K}(\varphi)
                  = \frac{\omega}2 \|\varphi\|_{L^2}^2 + \frac{d(p-1)-4}{4(p+1)} \|\varphi\|_{L^{p+1}}^{p+1}.
\end{split}
\end{equation}
Then we have an equivalent variational characterization of $m_\omega$.
\begin{proposition}[\cite{CMZ}] \label{le2.4}
\begin{equation}\label{eq2.4}
\begin{split}
m_\omega  = \inf \left\{ \mathcal{H}_\omega(\varphi): \varphi \in H^1\setminus \{0\}, \mathcal{K}(\varphi) \le 0\right\} .
\end{split}
\end{equation}
\end{proposition}
We also have
\begin{lemma}[\cite{CMZ}] \label{le2.6}
For any $ \varphi \in H^1(\mathbb{R}^d)$ with $\mathcal{K}(\varphi) \ge 0$, we have
\begin{equation*}
\begin{split}
\mathcal{E}(\varphi)  \sim  \int_{\mathbb R^d} \frac12|\nabla \varphi|^2 + \frac{d}{2(d+2)}
|\varphi|^\frac{2(d+2)}d \,\mathrm{d}x.
\end{split}
\end{equation*}
\end{lemma}
By a simple computation, we have $\forall\, \omega > 0$, $A_{\omega,+}$ is bounded in $H^1(\mathbb{R}^d)$.
\begin{lemma}\label{le2.8}
 Let $\omega > 0$ and $u \in A_{\omega,+}$, then we have
\begin{equation*} 
\left\|u\right\|_{H^1}^2 \lesssim m_\omega + \frac{m_\omega}\omega.
\end{equation*}
\end{lemma}
Next we give the energy-trapping properties of $A_{\omega,\pm}$.
\begin{proposition}[Energy-trapping for $A_{\omega,+}$, \cite{CMZ}]\label{le2.9}
For $u_0\in A_{\omega,+}$, let $u$ be the solution of \eqref{eq1.1}, there exists some positive constant $\delta  = \delta(d,p,\omega)$ such that for $t \in \mathbb{R}$,
\begin{equation*} 
\mathcal{K}(u(t)) \ge \min\left\{\frac{d(p-1)-4}{d(p-1)}\Big( \left\|\nabla u(t)\right\|_{L^2}^2 + \frac{d}{d+2} \left\|u(t)\right\|_{L_x^\frac{2(d+2)}d}^\frac{2(d+2)}d\Big), \delta
\Big(m_\omega - \mathcal{S}_\omega(u(t))\Big)\right\}.
\end{equation*}
\end{proposition}
\begin{proposition}[Energy-trapping for $A_{\omega,-}$, \cite{CMZ}]\label{le2.7}
 For $u_0\in A_{\omega,-}$, let $u$ be the solution of \eqref{eq1.1}, then
\begin{equation} \label{eq2.11}
\mathcal{K}(u(t)) < - \left(m_\omega - \mathcal{S}_\omega(u(t))\right), \   \forall \,t\in \mathbb{R}.
\end{equation}
\end{proposition}
As a consequence of this Proposition, we can give the proof of the blow-up result if $u_0 \in \Sigma$ in the non-radial case in Theorem \ref{th1.4}.
For the radial case, we refer to \cite{CMZ}.
For
\begin{equation*}
V(t) = \int_{\mathbb{R}^d}  |x|^2   |u(t,x)|^2  \,\mathrm{d}x,
\end{equation*}
by direct computation together with the fact $u_0 \in A_{\omega, -}$, energy, mass conservation and \eqref{eq2.11}, we have
\begin{equation*}
\begin{split}
V''(t) =8 \mathcal{K}(u)< - 8 \left(m_\omega - \mathcal{S}_\omega(u(t))\right),\ \forall\, t \in I_{max},
\end{split}
\end{equation*}
where $I_{max}$ is the maximal lifespan of $u$, which implies $u$ must blow up in finite time.

\section{Wellposedness and stability theory} \label{se4v24}
In this section, we present the wellposedness theory and the stability theory for \eqref{eq1.1}. For the proof we refer to \cite{Cazenave, Cazenave-Weissler, Killip-Visan1}, see also \cite{CMZ}.
\begin{proposition}\label{th4.1}
{\it (i)} (Local existence) Let $\phi \in H^{1}(\mathbb{R}^d)$, $I$ be an interval, $t_0\in I$ and $A>0$. Assume that
\begin{equation*}
\left\|\phi \right\|_{H^1}\le A,
\end{equation*}
and there exists $\delta>0$ depending on $A$ that
\begin{equation*}
\left\|\langle \nabla \rangle  e^{i(t-t_0)\Delta}\phi \right\|_{L_{t,x}^\frac{2(d+2)}d(I\times \mathbb{R}^d)} \le \delta,
\end{equation*}
then there exists a unique solution $u \in C(I,H^1(\mathbb{R}^{d}))$ with $u(t_0)= \phi$ to \eqref{eq1.1} such that
\begin{align*}
\left\| u \right\|_{S^1(I)} \lesssim   \left\| \phi \right\|_{H^1}, \text{ and }
\left\|  \langle \nabla \rangle u \right\|_{L_{t,x}^\frac{2(d+2)}d(I\times \mathbb{R}^d)} \le& 2 \left\| \langle \nabla\rangle e^{i(t-t_0)\Delta} \phi \right\|_{L_{t,x}^\frac{2(d+2)}d(I\times \mathbb{R}^d)}.
\end{align*}
As a consequence, we have the small data scattering:
if $\left\|\phi\right\|_{H^1}$ is sufficiently small, then $u$ is a global solution with $\left\|u\right\|_{S^1(\mathbb{R})} \lesssim \left\|\phi\right\|_{H^1}$.
By Lemma \ref{le2.8}, we see the solution is global in $H^1$ when $\phi\in A_{\omega,+}$.

{\it (ii)}(Scattering criterion)
Let $u \in C(\mathbb{R}, H^1(\mathbb{R}^d))$ be the solution to \eqref{eq1.1}, if
\begin{equation}\label{eq4.11}
\left\| \langle \nabla \rangle^{s_p} u \right\|_{L_{t,x}^\frac{2(d+2)}d (\mathbb{R} \times \mathbb{R}^d)} \le  L,
\end{equation}
for some positive constant $L$,
then there exist $u_{\pm}\in H^{1}(\mathbb{R}^d)$ such that
\begin{equation*}\label{eq4.12}
\lim_{t\to \infty}
\left\|u(t)-e^{it\Delta}u_+ \right\|_{H^1}=\lim_{t\to -\infty}\left\|u(t) - e^{it\Delta}u_- \right\|_{H^1}=0.
\end{equation*}
\end{proposition}
\begin{proof}
We refer to \cite{CMZ} for the proof of (i), due to the subcritical essential of the equation together with the variational estimate in Section \ref{se2v9}, the solution is global in time. For (ii), we only need to show \eqref{eq4.11} implies
\begin{align}\label{eq3.3v9}
\| u \|_{S^1(\mathbb{R})} < \infty.
 \end{align}
By the Strichartz estimate and the continuity method, we can easily obtain \eqref{eq4.11} implies
\begin{align}\label{eq3.4v9}
\| u \|_{S^{s_p}(\mathbb{R})} < C(L) .
\end{align}
By \eqref{eq3.4v9}, we can divide the time interval $\mathbb{R}$ into $N_1 \sim \left( 1 + \frac{C(L)} {\delta_0 } \right)^\frac{2(d+2)} d $ subintervals $J_k = [t_k, t_{k+1}]$ such that
\begin{align*}
\|u \|_{S^{s_p}(J_k)} \le \delta_0,
\end{align*}
where $\delta_0$ will be chosen later.
On each $J_k$, by Strichartz and H\"older, we have
\begin{align*}
\|u\|_{S^1(J_k)} \le C\left( \|u(t_k) \|_{H^1} + \delta_0^\frac4d \|   u \|_{S^1(J_k )} +\delta_0^{p-1} \|  u \|_{S^1(J_k \times \mathbb{R}^d)} \right).
\end{align*}
Therefore, by choosing $\delta_0 $ small enough, we have
\begin{align*}
\|u\|_{S^1(J_k)} & \le 2  C  \|u(t_k) \|_{H^1}.
\end{align*}
Summing over the subintervals $J_k$, we obtain \eqref{eq3.3v9}.
\end{proof}
\begin{remark}
In fact, we can show
\begin{equation*}
\left\|  u \right\|_{L_{t,x}^\frac{2(d+2)}d (\mathbb{R} \times \mathbb{R}^d)} \le  L,
\end{equation*}
is enough to show scattering by using the continuity method. But we find Proposition \ref{pr4.3} can not be proven if the approximate solution is only if $L_{t,x}^\frac{2(d+2)}d$, therefore we use the norm $L_t^\frac{2(d+2)}d W_x^{s_p, \frac{2(d+2)}d}$.
\end{remark}
In the following, we will give the stability theory.
\begin{proposition}[Stability theory]\label{pr4.3}
Let $I$ be a compact time interval and let $w$ be an approximate solution to \eqref{eq1.1} on $I\times \mathbb{R}^d$ in the sense that
\begin{equation*}
i\partial_t w + \Delta w = |w|^\frac4d w - |w|^{p-1} w + e
\end{equation*}
for some function $e$. Assume that
\begin{align*}
\left\|w\right\|_{L_t^\infty H_x^1(I\times \mathbb{R}^d)} \le A_1,  \
\left\|\langle \nabla \rangle^{s_p} w\right\|_{L_{t,x}^\frac{2(d+2)}d (I\times \mathbb{R}^d)} \le B \label{chuzjias}
\end{align*}
for some $A_1, B > 0$.

Let $t_0\in I$ and $u(t_0)$ close to $w(t_0)$ in the sense that
\begin{equation*}\label{eq1.34v11}
\left\|u(t_0) - w(t_0)\right\|_{H_x^{s_p}(\mathbb{R}^d)} \le A_2
\end{equation*}
for some $A_2 > 0$. Assume also the smallness conditions
\begin{align}
\left\|\langle \nabla \rangle^{s_p} e^{i(t-t_0)\Delta}(u(t_0)- w(t_0))\right\|_{L_{t,x}^\frac{2(d+2)}d(I\times \mathbb{R}^d)} +  
\left\|\langle \nabla \rangle^{s_p} e\right\|_{L_{t,x}^\frac{2(d+2)}{d+4}(I\times \mathbb{R}^d)} \le \delta \label{eq3.7v37}
\end{align}
for some $0 < \delta \le \delta_1$, where $\delta_1 = \delta_1(A_1, A_2, B)$ is a small constant. Then there exists a solution $u$ to \eqref{eq1.1} on
$I\times \mathbb{R}^d$ with the specified initial data $u(t_0)$ at time $t = t_0$ that satisfies
\begin{align*}
\left\|\langle \nabla \rangle^{s_p} (  u - w) \right\|_{L_{t,x}^\frac{2(d+2)}d ( I \times \mathbb{R}^d)} \le  C(A_1, A_2, B) \delta, \\
\left\|u - w\right\|_{S^{s_p}(I)} \le  C(A_1, A_2, B) A_2, \quad
\left\| u\right\|_{S^{s_p}(I)} \le  C(A_1, A_2, B).
\end{align*}
\end{proposition}

\section{Linear profile decomposition}\label{se5v24}
The linear profile decomposition was first established by S. Keraani \cite{Keraani1} for the Schr\"odinger equation in $\dot{H}^1(\mathbb{R}^d)$. At almost the same time, F. Merle and L. Vega \cite{Merle-Vega} gave the linear profile decomposition for the Schr\"odinger equation in $L^2(\mathbb{R}^2)$, then R. Carles and S. Keraani \cite{Carles-Keraani} established this in $L^2(\mathbb{R})$. S. Keraani \cite{Keraani2} used the linear profile decomposition in $L^2$ to describe the minimal mass blowup solution of the mass-critical nonlinear Schr\"odinger equation. Later, P. B\'egout and A. Vargas \cite{Begout-Vargas} proved the linear profile decomposition in $L^2(\mathbb{R}^d),\  d \ge 3$ of the Schr\"odinger equation. In this section, we will give the linear profile decomposition in $H^1(\mathbb{R}^d)$ by using the linear profile decomposition in $L^2(\mathbb{R}^d)$ for the Schr\"odinger equation. First, we review the linear profile decomposition in $L^2(\mathbb{R}^d)$.
\begin{lemma}[Profile decomposition in $L^2(\mathbb{R}^d)$, $d\ge 1$,\cite{Begout-Vargas,Carles-Keraani,Keraani2,Merle-Vega,Tao-Visan-Zhang}]\label{th4.1v7}
Let $\{\varphi_n\}_{n\ge 1}$ be a bounded sequence in $L^2(\mathbb{R}^d)$. Then up to passing to a subsequence of $\{\varphi_n\}_{n\ge 1}$, there exists a sequence of functions $\varphi^j \in L^2(\mathbb{R}^d)$ and $( \theta_n^j, h_n^j, t_n^j, x_n^j, \xi_n^j)_{n\ge 1} \subseteq  \mathbb{R}/2\pi\mathbb{Z} \times (0,\infty)\times  \mathbb{R} \times \mathbb{R}^d \times \mathbb{R}^d$, with
\begin{align}\label{eq8.4v14}
& \frac{h_n^l}{h_n^j} + \frac{h_n^j}{h_n^l} + \frac{|t_n^j -t_n^l|}{(h_n^j)^2}+ h_n^j | \xi_n^j -\xi_n^l|  +  \frac{\left| x_n^j- x_n^l  + 2t_n^j (\xi_n^j -   \xi_n^l)\right| }{h_n^j} \to \infty,  \text{  for }  j \ne l,\\
& h_n^j \to h_\infty^j\in \{0, 1, \infty\},\  h_n^j = 1 \text{ if }  h_\infty^j = 1, \label{eq8.1v14}\\
& \tau_n^j := - \frac{t_n^j}{(h_n^j)^2} \to \tau_\infty^j \in [-\infty, \infty], \text{ as }  n\to \infty, \label{eq8.2v14}\\
\text{ and } & \xi_n^j = 0 \text{  if  }    \underset{ n \to \infty} {\limsup}\, |h_n^j \xi_n^j| < \infty, \label{eq8.3v14}
\end{align}
such that $\forall \,k \ge 1$, there exists $w^k_n\in L^2(\mathbb{R}^d)$,
\begin{equation*}
\varphi_n(x) = \sum\limits_{j=1}^k  T_n^j \varphi^j(x) + w^k_n(x),
\end{equation*}
where $T_n^j$ is defined by
\begin{align}\label{eq4.5v31}
T_n^j \varphi(x) := e^{i\theta_n^j}
e^{ix\cdot\xi_n^j} e^{-it_n^j \Delta} \left(\frac1{(h_n^j)^\frac{d}2 } \varphi \left(\frac{\cdot - x_n^j }{h_n^j}\right)\right)(x).
\end{align}
The remainder $w_n^k$ satisfies
\begin{equation}\label{eq8.6v14}
 \underset{ n\to \infty} \limsup \, \left\|e^{it\Delta} w_n^k\right\|_{L_t^q L_x^r(\mathbb{R}\times \mathbb{R}^d)} \to 0, \text{ as }  k\to \infty,
\end{equation}
 where $(q,r)$ is $L^2$-admissible, and $2 < q < \infty \text{ when } d\ge 2,\  4 < q < \infty \text{ when } d = 1$.
 We also have as $n\to \infty$,
\begin{align}
& \ \left\|e^{it\Delta} T_n^j \varphi^j   \cdot e^{it\Delta} T_n^l \varphi^l \right\|_{L_{t,x}^\frac{d+2}d(\mathbb{R} \times \mathbb{R}^d)} \to 0, \  \langle T_n^j \varphi^j, T_n^l \varphi^l \rangle_{L^2}\to 0,  \text{ for } j\ne l, \nonumber \\
  \text{ and } &   \langle T_n^j \varphi^j , w_n^k\rangle_{L^2} \to 0, \  (T_n^j)^{-1} w_n^k \rightharpoonup 0 \text{ in } L^2(\mathbb{R}^d),\ \forall \,1 \le j\le k.\label{eq8.8v14}
\end{align}
As a consequence, we have the mass decoupling property:
\begin{equation}\label{eq8.9v14}
\forall\, k \ge 1,\
\left\|\varphi_n\right\|_{L^2}^2 - \sum\limits_{j=1}^k \left\|\varphi^j\right\|_{L^2}^2 - \left\|w_n^k\right\|_{L^2}^2 \to 0, \text{ as } n\to \infty.
\end{equation}
\end{lemma}
\begin{proof}
We only need to show \eqref{eq8.1v14}, \eqref{eq8.6v14} and \eqref{eq8.8v14}. Other statements in the theorem are stated in the profile decomposition in $L^2(\mathbb{R}^d)$ proved in \cite{Begout-Vargas, Carles-Keraani, Keraani2, Merle-Vega,Tao-Visan-Zhang}. Without loss of generality, we assume that the sequence is up to a subsequence in the following.

To show \eqref{eq8.1v14}, we only need to prove that we may take $h_\infty^j$ and $h_n^j$ to be 1 when $h_\infty^j \in (0, \infty)$. In fact,
if $h_n^j \to h_\infty^j \in (0, \infty), \text{  as }  n\to \infty$, by the Strichartz estimate, we can put $e^{i\theta_n^j} e^{ix\cdot\xi_n^j} e^{-it_n^j \Delta} \left(\frac1{(h_n^j)^\frac{d}2 } \varphi^j \left(\frac{\cdot - x_n^j }{h_n^j}\right)\right)(x)- e^{i\theta_n^j} e^{ix\cdot\xi_n^j} e^{-it_n^j \Delta} \left(\frac1{(h_\infty^j)^\frac{d}2 } \varphi^j
\left(\frac{\cdot - x_n^j }{h_\infty^j}\right)\right)(x)$ into the remainder term. We now shift $\varphi^j(x)$ by
$\frac1{(h_\infty^j)^\frac{d}2} \varphi^j\left(\frac{x}{h_\infty^j}\right)$, and $(\theta_n^j, h_n^j, t_n^j, x_n^j, \xi_n^j)$ by $(\theta_n^j,1, t_n^j, x_n^j, \xi_n^j)$. It is easy to see that \eqref{eq8.4v14}-\eqref{eq8.3v14} are not affected. Thus, we
conclude \eqref{eq8.1v14}.

We now show \eqref{eq8.3v14}.
If $h_n^j \xi_n^j \to \xi^j\in \mathbb{R}^d, \text{ as }  n\to \infty, \text{ for some }  1 \le j \le k$. By the Galilean transform and Strichartz estimate, we can replace
$( \theta_n^j, h_n^j, t_n^j, x_n^j, \xi_n^j)$ by $(\theta_n^j + t_n^j |\xi_n^j|^2 + x_n^j\cdot \xi_n^j, h_n^j, t_n^j, x_n^j + 2t_n^j \xi_n^j, 0)$, and $\varphi^j(x)$ by $e^{ix\cdot\xi^j} \varphi^j(x)$, and we can verify \eqref{eq8.4v14}-\eqref{eq8.3v14} are not affected.
So we can take $\xi_n^j = 0$ when $\underset{n\to \infty} \limsup \, |h_n^j \xi_n^j | < \infty$.

For the profile decomposition in \cite{Begout-Vargas, Carles-Keraani, Merle-Vega}, the remainder $w_n^k$ satisfies
\begin{equation*}
\underset{n\to \infty} \limsup \, \left\|e^{it\Delta} w_n^k\right\|_{L_{t,x}^\frac{2(d+2)}d(\mathbb{R} \times \mathbb{R}^d)} \to 0, \text{ as }  k\to \infty.
 \end{equation*}
By using interpolation, the Strichartz estimate and \eqref{eq8.9v14}, we easily obtain \eqref{eq8.6v14}.
\end{proof}
We can now show the linear profile decomposition in $H^1(\mathbb{R}^d)$, which reveals the defect of compactness of the embedding $e^{it\Delta_{\mathbb{R}^d}}: H_x^1(\mathbb{R}^d)  \hookrightarrow L_{t }^\frac{2(d+2)} d W_x^{s_p, \frac{2(d+2)}d} (\mathbb{R}\times \mathbb{R}^d)$.
\begin{theorem}[Profile decomposition in $H^1(\mathbb{R}^d)$]\label{le5.2}
Let $\{\varphi_n\}_{n\ge 1}$ be a bounded sequence in $H^1(\mathbb{R}^d)$. Then up to passing to a subsequence of $\left\{\varphi_n\right\}_{n\ge 1}$, there exists a sequence of functions $\varphi^j \in L^2(\mathbb{R}^d)$, $w_n^k \in H^1(\mathbb{R}^d)$, and $\left( \theta_n^j, h_n^j, t_n^j, x_n^j, \xi_n^j\right)_{n\ge 1} \subseteq  \mathbb{R}/2\pi\mathbb{Z} \times (0,\infty)\times  \mathbb{R} \times \mathbb{R}^d \times \mathbb{R}^d$, $| \xi_n^j | \le C_j$, with 
  \begin{align*}
\frac{h_n^l}{h_n^j} + \frac{h_n^j}{h_n^l} + \frac{|t_n^j - t_n^l|}{(h_n^j)^2}+ h_n^j | \xi_n^j -\xi_n^l|  + \frac{ \left| x_n^j- x_n^l +  2t_n^j( \xi_n^j -  \xi_n^l) \right| }{h_n^j} \to \infty, \text{ as } n\to \infty,  \forall\, j \ne l,    
\end{align*}
and
\begin{align*}
 &  h_n^j \to h_\infty^j \in \{1, \infty\},\  h_n^j = 1 \text{ if  } h_\infty^j = 1,
 \end{align*}
such that for any $k \in \mathbb{N}$, we have the decomposition 
\begin{equation*}
\begin{split}
\varphi_n(x) = \sum\limits_{j=1}^k T_n^j P_n^j \varphi^j(x) + w_n^k(x),
\end{split}
\end{equation*}
where $T_n^j$ is defined in \eqref{eq4.5v31}, and the projector $P_n^j$ is defined by
\begin{align*}
P_n^j  \varphi^j(x) =
\begin{cases}
\varphi^j(x),                      & \text{ if } h_n^j  \equiv 1,\\
P_{\le (h_n^j)^\theta}  {\varphi}^j, 0 < \theta < 1,  & \text{ if } h_n^j \to \infty.
\end{cases}
\end{align*}
Furthermore, if $h_n^j = 1$, we can take $\xi_n^j = 0$ and $\varphi^j \in H^1$. The remainder $w_n^k$ satisfies
\begin{equation*}
\underset{n\to \infty}{\limsup}\  \left\| \langle \nabla \rangle^{s_p} e^{it\Delta} w_n^k\right\|_{L_t^q L_x^r(\mathbb{R} \times \mathbb{R}^d)} \to 0, \text{ as }  k \to \infty,
\end{equation*}
where $(q,r)$ is $L^2$-admissible, and $2 < q < \infty \text{ when } d\ge 2,\, 4 < q < \infty \text{ when }  d = 1$.
Moreover, we have the following decoupling properties: $\forall \, k\in \mathbb{N}$,
\begin{align*}
& \big\||\nabla|^s \varphi_n\big\|_{L^2}^2 - \sum\limits_{j=1}^k \left\||\nabla|^s \left(T_n^j P_n^j \varphi^j \right) \right\|_{L^2}^2 - \big\||\nabla|^s w_n^k\big\|_{L^2}^2 \to 0,  \ s = 0,1, 
\\
& \mathcal{E}(\varphi_n) - \sum\limits_{j=1}^k \mathcal{E}(T_n^j P_n^j \varphi^j ) - \mathcal{E}(w_n^k) \to 0, 
  \\
& \mathcal{S}_\omega(\varphi_n) - \sum\limits_{j=1}^k \mathcal{S}_\omega(T_n^j P_n^j  \varphi^j) - \mathcal{S}_\omega ( w_n^k) \to 0,
 \\
& \mathcal{K}(\varphi_n) - \sum\limits_{j=1}^k \mathcal{K}(T_n^j P_n^j  \varphi^j) - \mathcal{K}(w_n^k) \to 0, 
\\
& \mathcal{H}_\omega(\varphi_n) - \sum\limits_{j=1}^k \mathcal{H}_\omega( T_n^j P_n^j \varphi^j) - \mathcal{H}_\omega(w_n^k) \to 0, \text{  as } n\to \infty.  
\end{align*}
\end{theorem}
{ \it Sketch of the proof.}
By interpolation and Strichartz estimate, we have
\begin{align*}
\left\| \langle \nabla\rangle^{s_p} e^{it \Delta} w_n^k \right\|_{L_{t,x}^\frac{2(d+2)}d}
& \lesssim \left\| e^{it \Delta} w_n^k \right\|_{L_{t,x}^\frac{2(d+2)}d}^{1- s_p} \left\| \langle \nabla \rangle e^{it \Delta} w_n^k \right\|_{L_{t,x}^\frac{2(d+2)}d}^{s_p}\\
& \lesssim \left\| e^{it\Delta} w_n^k \right\|_{L_{t,x}^\frac{2(d+2)}{d}}^{1- s_p} \left\|w_n^k \right\|_{H_x^1}^{s_p}.
\end{align*}
Therefore we can reduce the proof to the description the defect of compactness of the embedding $e^{it\Delta}: H_x^1 \hookrightarrow L_{t,x}^\frac{2(d+2)}d$. The argument follows from the proof of \cite{CGYZ,CGZ}.

\section{Extraction of a critical element}\label{se6v24}
In this section, we show the existence of a critical element which is almost periodic modulo spatial translation in $H^1(\mathbb{R}^d)$ by using the linear profile decomposition and the stability theory.

As a consequence of the scattering criterion in Proposition \ref{th4.1}, we can reduce the proof of Theorem \ref{th1.4} to the proof of finiteness of the space-time norm $L_t^\frac{2(d+2)}d W_x^{s_p, \frac{2(d+2)}d}$ of any solution $u$ to \eqref{eq1.1} with $u_0 \in A_{\omega, +}$. For any $m > 0$, let
\begin{equation*}
\Lambda_\omega(m) = \sup \left\| \langle \nabla \rangle^{s_p}  u\right\|_{L_{t,x}^\frac{2(d+2)}d ( \mathbb{R} \times \mathbb{R}^d)} ,
\end{equation*}
where the supremum is taken over the solution $u$ to \eqref{eq1.1} with $u_0 \in A_{\omega, +}$ and $\mathcal{S}_\omega(u_0) \le m$,
and define
\begin{equation} \label{eq5.2}
m_\omega^* = \sup\{ m>0: \Lambda_\omega(m)< \infty\}.
\end{equation}
 If  $u_0\in A_{\omega, +}$ with  $\mathcal{S}_\omega(u_0) \le m$ sufficiently small, then Lemma \ref{le2.8} shows
$\|u\|_{H^1} \ll 1$. Hence, Proposition \ref{th4.1}{(\textit{i})} gives the finiteness of $\Lambda_\omega(m)$, which implies $m_\omega^* > 0$, and our aim is to show $m_\omega^* \ge  m_\omega$.

Suppose by contradiction that $m_\omega^* < m_\omega$, we can then take a sequence $\{u_n\}_{n\ge 1}$ of solutions (up to time translations) to \eqref{eq1.1} such that
\begin{align}
& u_n(t) \in A_{\omega,+}, \text{ for $t\in \mathbb{R}$,}  \  \text{ and } \mathcal{S}_\omega(u_n) \to m_\omega^*, \text{  as $n\to \infty$,} \label{eq5.4}\\
& \underset{n\to \infty} {\lim} \|\langle \nabla \rangle^{s_p}  u_n\|_{L_{t,x}^\frac{2(d+2)}d ([0,\infty) \times \mathbb{R}^d)} = \underset{n\to \infty} {\lim} \| \langle \nabla \rangle^{s_p}   u_n\|_{L_{t,x}^\frac{2(d+2)}d ((- \infty , 0]\times \mathbb{R}^d)} = \infty. \notag
\end{align}
By Lemma \ref{le2.8},
\begin{equation}\label{eq5.6}
\underset{n}{\sup}\,  \left\|u_n\right\|_{L_t^\infty H_x^1(\mathbb{R} \times \mathbb{R}^d)}^2 \lesssim m_\omega + \frac{m_\omega}\omega.
\end{equation}
Applying Theorem \ref{le5.2} to $\left\{u_n(0)\right\}_{n \ge 1}$ and obtain some subsequence of $\left\{u_n(0)\right\}_{n\ge 1}$ (still denoted by the same symbol), then
there exists $ \varphi^j \in L^2(\mathbb{R}^d)$, $w_n^k\in H^1(\mathbb{R}^d)$, and $\left( \theta_n^j, h_n^j, t_n^j, x_n^j, \xi_n^j\right)_{n\ge 1} $ of sequences in $   \mathbb{R}/2\pi\mathbb{Z} \times (0,\infty)\times  \mathbb{R} \times \mathbb{R}^d \times \mathbb{R}^d$, with
  \begin{align}
 &  \tau_n^j :=  - \frac{t_n^j}{(h_n^j)^2} \to \tau_\infty^j \in [-\infty, \infty], \nonumber \\
 &  h_n^j \to h_\infty^j \in \{  1, \infty\},\  h_n^j = 1 \text{ if  } h_\infty^j = 1, \nonumber \\
&  \frac{h_n^l}{h_n^j} + \frac{h_n^j}{h_n^l} + \frac{|t_n^j -t_n^l|}{(h_n^j)^2}+ h_n^j | \xi_n^j -\xi_n^l|  +  \frac{\left| x_n^j- x_n^l + 2t_n^j ( \xi_n^j -  \xi_n^l)
  \right|}{h_n^j}   \to \infty, \ \forall\, j \ne l,  \label{eq6.48v3}
\end{align}
$\text{ as }  n\to \infty$, such that $\forall \,k \in \mathbb{N}$, we have
\begin{align}\label{eq6.51v3}
e^{it\Delta} u_n(0,x) & = \sum\limits_{j=1}^k e^{it\Delta} T_n^j P_n^j \varphi^j(x) + e^{it\Delta} w_n^k(x).
\end{align}
The remainder $w_n^k$ satisfies
\begin{equation}\label{eq6.50v3}
\underset{n\to \infty}{\limsup}\  \left\| \langle \nabla \rangle^{s_p}  e^{it\Delta} w_n^k\right\|_{L_t^q L_x^r(\mathbb{R} \times \mathbb{R}^d)} \to 0, \text{ as }  k \to \infty,
\end{equation}
where $(q,r)$ is $L^2$-admissible, $2< q < \infty$ when $d\ge 2$ and $4 < q < \infty$ when $d = 1$.
Moreover, we have
\begin{align}
& \left\| |\nabla|^s u_n(0)\right\|_{L^2}^2 - \sum\limits_{j=1}^k \left\||\nabla|^s  \left(T_n^j P_n^j \varphi^j \right)  \right\|_{L^2}^2 - \||\nabla|^s w_n^k\|_{L^2}^2 \to 0,\  \forall \, s = 0, 1, \label{eq6.52v3}\\
& \mathcal{S}_\omega(u_n(0)) - \sum\limits_{j=1}^k \mathcal{S}_\omega\Big(T_n^j P_n^j \varphi^j \Big) - \mathcal{S}_\omega ( w_n^k) \to 0,
\notag
\\
& \mathcal{H}_\omega(u_n(0)) - \sum\limits_{j=1}^k \mathcal{H}_\omega\Big(T_n^j P_n^j \varphi^j \Big) - \mathcal{H}_\omega(w_n^k) \to 0, 
\notag
\\
& \mathcal{E}(u_n(0)) - \sum\limits_{j=1}^k \mathcal{E}\Big(T_n^j P_n^j \varphi^j   \Big) - \mathcal{E}(w_n^k) \to 0, \nonumber\\
& \mathcal{K}(u_n(0)) - \sum\limits_{j=1}^k \mathcal{K}\Big(T_n^j P_n^j \varphi^j \Big) - \mathcal{K}(w_n^k) \to 0, \text{  as } n\to \infty.\nonumber
\end{align}
Using Strichartz estimate, \eqref{eq6.52v3} and \eqref{eq5.6},  we get
\begin{equation*}
\underset{k\in \mathbb{N}} \sup \ \underset{n\to \infty} \limsup\, \left\|  e^{it\Delta} w_n^k\right\|_{ S^1(\mathbb{R})}
\lesssim \underset{k \in \mathbb{N}}\sup\   \underset{n \to \infty} \limsup \,\left\|w_n^k\right\|_{H^1} < \infty.
\end{equation*}
Next, we construct the nonlinear profile. Let $u_n^j$ be the solution to \eqref{eq1.1} with $u_n^j(0,x) = T_n^j P_n^j \varphi^j$, then for $h_n^j \to \infty$, as $n\to \infty$, we have the following approximate theorem.
\begin{theorem}[Approximation of the large scale profiles]\label{th6.1v26}
For any $\varphi \in L_x^2$, $0 < \theta < 1$, $( \theta_n, h_n, t_n, x_n, \xi_n)_{n\ge 1} \subseteq  \mathbb{R}/2\pi\mathbb{Z} \times (0,\infty)\times  \mathbb{R} \times \mathbb{R}^d \times \mathbb{R}^d$, $h_n \to \infty$, as $n\to \infty$, $| \xi_n| \lesssim 1$, there is a global solution $u_n$ of \eqref{eq1.1} with $u_n(0,x) = T_n P_n \varphi$, for $n$ large enough satisfying
\begin{align*}
\left\|u_n \right\|_{L_t^\infty H_x^{s_p} \cap L_t^\frac{2(d+2)}d W_x^{s_p, \frac{2(d+2)}d}(\mathbb{R} \times \mathbb{R}^d)} \lesssim_{\|\varphi\|_{L_x^2}} 1.
\end{align*}
Furthermore, assume $\epsilon_1$ is a sufficiently small positive constant depending only on $\left\|\varphi\right\|_{L^2}$, $v_0\in H^1_x$, and
\begin{align*}
\left\|\varphi - v_0 \right\|_{L_x^2} \le \epsilon_1.
\end{align*}
There exists a solution $v \in C_t^0 H_x^1$ to
\begin{align*}
\begin{cases}
i\partial_t v + \Delta v = |v|^\frac4d v ,\\
v(0) = v_0,
\end{cases}
\end{align*}
with
\begin{align*}
 v(0)  = v_0,  & \text{ if } t_n = 0,\\
\lim\limits_{t\to \pm \infty} \left\|v(t) - e^{it \Delta} v_0 \right\|_{L^2}  \to 0,    & \text{ if } - \frac{t_n}{h_n^2} \to \pm \infty,
\end{align*}
such that for any $\epsilon > 0$, it holds that
\begin{align*}
& \left\|u_n (t) - w_n(t) \right\|_{L_t^\infty H_x^{s_p} \cap L_t^\frac{2(d+2)}d W_x^{s_p, \frac{2(d+2)}d}} \lesssim_{\|\varphi\|_{L^2}} \epsilon_1, \\
& \left\|u_n\right\|_{L_t^\infty H_x^{s_p} \cap L_t^\frac{2(d+2)}d W_x^{s_p, \frac{2(d+2)}d}} \lesssim 1,
\end{align*}
for $n$ large enough, where
\begin{align*}
w_n(t,x) =  \frac1{h_n^\frac{d}2} e^{-i t   |\xi_n|^2} e^{i x \xi_n} v\left(\frac{t - t_n }{h_n^2} , \frac{x- x_n - 2\xi_n t }{h_n}\right).
\end{align*}
\end{theorem}
\begin{proof}
We just give a sketch of the proof. Without loss of generality, we may assume that $x_n = 0$, using a Galilean transform and the fact that $\xi_n$ is bounded, we may assume that $\xi_n = 0$ for all $n$.

When $t_n = 0$, we can show $w_n$ is an approximate solution to the equation \eqref{eq1.1} in the sense of the stability theory in Proposition \ref{pr4.3}, and we only need to verify the latter estimate in \eqref{eq3.7v37}.
We can see
\begin{align*}
e_{n} := \left( i \partial_t + \Delta\right) w_n - |w_n|^\frac4d w_n +  |w_n|^{p-1} w_n = (h_n)^{- \frac{dp}2} \left( |v|^{p-1} v\right)\left(\frac{t}{h_n^2}, \frac{x}{h_n}\right),
\end{align*}
then by the fractional chain rule and H\"older's inequality, we have
\begin{align*}
\left\|\langle \nabla \rangle^{s_p} e_{n} \right\|_{L_{t,x}^\frac{2(d+2)}{d+4}}
\lesssim  & \  h_n^\frac{d+4 - dp}2  \|  v\|_{L_{t,x}^\frac{(d+2)(p-1)}{2}}^{p-1} \|   v\|_{L_{t,x}^\frac{2(d+2)}{d}}\\
  &   + h_n^{\frac{d+4 - dp}2 - s_p}  \| v\|_{L_{t,x}^\frac{(d+2)(p-1)}{2}}^{p-1}  \left\| \left|\nabla\right|^{s_p}  v \right\|_{L_{t,x}^\frac{2(d+2)}{d}},
\end{align*}
which is small for $n$ large enough, because $h_n\to \infty$, as $n\to \infty$ and $\|v\|_{S^1(\mathbb{R})} \le   C(\|v_0\|_{H^1})$, which is a consequence of persistence of regularity and the scattering result of the mass-critical nonlinear Schr\"odinger equation \cite{D1,D2,D3,Killip-Tao-Visan, Killip-Visan-Zhang,Tao-Visan-Zhang2,Tao-Visan-Zhang3}.

When $- \frac{t_n}{h_n^2}  \to 0$, as $n \to \infty$, $v$ is a solution of the mass-critical nonlinear Schr\"odinger equation with
$ \left\|v(t) - e^{it \Delta} v_0 \right\|_{L^2}  \to 0$, as $t \to \infty$. By the argument in the previous case, we can also obtain the result. Similarly, we can obtain the result when $\frac{t_n}{h_n^2} \to 0$, as $n\to \infty$.
\end{proof}
For the linear profile decomposition \eqref{eq6.51v3},
we can give the corresponding nonlinear profile decomposition
\begin{equation*} 
u_n^{< k}(t) = \sum\limits_{j=1}^k u_n^j(t),
\end{equation*}
where $u_n^j$ is the solution of
\begin{equation*}
\begin{cases}
( i\partial_t  + \Delta) u_n^j    = |u_n^j|^\frac4d u_n^j - |u_n^j|^{p-1} u_n^j,\\
    u_n^j(0,x) = T_n^j P_n^j \varphi^j,
\end{cases}
\end{equation*}
we will show $u_n^{< k} + e^{it\Delta} w_n^k$ is a good approximation for $u_n$ when $n$ large enough provided that each
nonlinear profile has finite global Strichartz norm, which is the key to show the existence of the critical element in the following proposition.
\begin{proposition}[Existence of the critical element] \label{pr5.8}
Suppose $m_\omega^* < m_\omega$, then there exist a global solution $u_c\in C(\mathbb{R}, H^1(\mathbb{R}^d))$ to \eqref{eq1.1} such that
\begin{align}
& u_c(t) \in A_{\omega, +}  \text{ and } \mathcal{S}_\omega(u_c(t)) =  m_\omega^*,\text{    for }  t\in \mathbb{R}, 
\notag
\\
& \| \langle \nabla \rangle^{s_p} u_c\|_{   L_{t,x}^\frac{2(d+2)}d([0,\infty)\times \mathbb{R}^d)} = \| \langle \nabla \rangle^{s_p} u_c\|_{L_{t,x}^\frac{2(d+2)}d((-\infty,0]\times \mathbb{R}^d)} = \infty. \label{eq6.161v3}
\end{align}
\end{proposition}
\begin{proof}
We can assume $t_n = 0$ by replacing $u_n(t)$ with $u_n(t+t_n)$. Applying the linear profile decomposition to $\left\{u_n(0,x)\right\}_n$, we have, after passing to a subsequence,
\begin{align*}
u_n(0,x) = \sum\limits_{j = 1}^k T_n^j P_n^j \varphi^j(x) + w_n^k(x).
\end{align*}
The remainder has asymptotically trivial linear evolution
\begin{align*}
\limsup \limits_{n\to \infty} \left\| \langle \nabla \rangle^{s_p} e^{it\Delta} w_n^k \right\|_{L_{t,x}^\frac{2(d+2)}d} \to 0, \text{ as } k \to \infty,
\end{align*}
and we also have asymptotic decoupling of the mass, energy and other functionals:
\begin{align}
\mathcal{M}(u_n(0)) - \sum\limits_{j=1}^k \mathcal{M}(T_n^j P_n^j \varphi^j ) - \mathcal{ M}(w_n^k) \to 0, \nonumber \\
\mathcal{E}(u_n(0)) - \sum\limits_{j=1}^k  \mathcal{E}(T_n^j P_n^j \varphi^j ) - \mathcal{E}(w_n^k) \to 0, \nonumber \\
\mathcal{S}_\omega(u_n(0)) - \sum\limits_{j=1}^k \mathcal{S}_\omega(T_n^j P_n^j \varphi^j ) - \mathcal{S}_\omega(w_n^k) \to 0, \nonumber \\
\mathcal{K}(u_n(0)) - \sum\limits_{j=1}^k  \mathcal{K}(T_n^j P_n^j \varphi^j ) - \mathcal{ K}(w_n^k) \to 0, \nonumber \\
\mathcal{H}_\omega(u_n(0)) - \sum\limits_{j=1}^k \mathcal{H}_\omega(T_n^j P_n^j \varphi^j ) - \mathcal{ H}_\omega(w_n^k) \to 0,\label{eq6.170v3}
\end{align}
as $n\to \infty$.
 Since $\mathcal{K}(u_n(0)) > 0$, we have $\mathcal{H}_\omega(u_n(0)) \le \mathcal{S}_\omega(u_n(0))$  by \eqref{eq2.3new}. It follows from \eqref{eq6.170v3} and \eqref{eq5.4} that
\begin{equation*}
\mathcal{H}_\omega( T_n^j P_n^j \varphi^j),\  \mathcal{H}_\omega(w_n^J) \le \mathcal{S}_\omega(u_n(0)) < \frac{m_\omega + m_\omega^*}2,
\end{equation*}
for $1\le j \le J$ and $n$ large enough. By \eqref{eq2.4}, we have
\begin{equation*}
\mathcal{K}\left(T_n^j P_n^j \varphi^j\right) > 0,\  \mathcal{K}(w_n^J) > 0, \text{  for } 1\le j \le J \text{ and $n$ large enough,}
\end{equation*}
which together with Lemma \ref{le2.6} shows
\begin{align*}
 \mathcal{S}_\omega\left( T_n^j P_n^j   \varphi^j\right) \ge  0, \ \mathcal{S}_\omega( w_n^J ) \ge 0,
\end{align*}
for $ 1 \le j \le J$ and $n$ large enough. There are two possibilities:

{\it Case 1.} $\sup\limits_{j } \limsup\limits_{n\to \infty} \mathcal{S}_\omega\left( T_n^j P_n^j \varphi^j\right) = m_\omega^*$. We deduce that
\begin{align*}
u_n(0,x) =   e^{i\theta_n } e^{ix \xi_n } e^{-it_n  \Delta} \left( \frac1{h_n^\frac{d}2} (P_n  \varphi )\left( \frac{\cdot - x_n }{h_n }\right) \right) (x) + w_n (x),
\end{align*}
with $\lim\limits_{n\to \infty} \left\|w_n\right\|_{H^1} = 0$. If $h_n \to \infty$, by Theorem \ref{th6.1v26}, there exists a unique global solution $u_n$ for $n$ large enough with
\begin{align*}
\limsup\limits_{n\to \infty} \left\| \langle \nabla \rangle^{s_p} u_n\right\|_{L_{t,x}^\frac{2(d+2)}d(\mathbb{R}\times \mathbb{R}^d)} \le C(m_\omega^*),
\end{align*}
which is a contradiction with \eqref{eq6.161v3}. Therefore, $h_n = 1$, and
\begin{align}\label{eq6.183v3}
u_n(0,x) =   \varphi  \left(  {x - x_n } \right)  + w_n (x),
\end{align}
 this is precisely the conclusion. If $t_n \to  \infty$, by the Galilean transform,
we observe
\begin{align*}
& \left\| \langle \nabla \rangle^{s_p}  e^{it\Delta} ( e^{i \theta_n} e^{ix \xi_n} e^{-it_n \Delta}   \varphi ( x - x_n) \right\|_{L_{t,x}^\frac{2(d+2)}d((-\infty, 0) \times \mathbb{R}^d)}\\
\lesssim  & \ \langle \xi_n\rangle^{s_p}   \left\|    e^{it \Delta}      \varphi ( x - x_n)   \right\|_{L_{t,x}^\frac{2(d+2)}d((-\infty, - t_n) \times \mathbb{R}^d)} \to 0 , \text{ as } n \to \infty.
\end{align*}
As a consequence of the local well-posedness, we see for $n$ large enough,
\begin{align*}
\|\langle \nabla \rangle^{s_p} u_n\|_{L_{t,x}^\frac{2(d+2)}d ((-\infty, 0) \times \mathbb{R}^d)} \le 2 \delta_0,
\end{align*}
which contradicts \eqref{eq6.161v3}. The case $t_n \to - \infty$ is similar.

{\it Case 2.} $\sup\limits_j \limsup\limits_{n\to \infty} \mathcal{S}_\omega( T_n^j P_n^j \varphi^j) \le m_\omega^* - 2 \delta$ for some $\delta >0$. By the definition of $m_\omega^*$, the nonlinear profile $u_n^j$ satisfies $\left\| \langle \nabla \rangle^{s_p} u_n^j \right\|_{L_{t,x}^\frac{2(d+2)}d(\mathbb{R}\times \mathbb{R}^d)} \lesssim \Lambda_\omega\left(m_\omega^* - \delta\right) < \infty$. We can then deduce that
\begin{align}\label{eq5.28v15}
\left\|\langle \nabla \rangle^{s_p} u_n^j \right\|^2_{L_{t,x}^\frac{2(d+2)}d(\mathbb{R}\times \mathbb{R}^d)} \lesssim_{m_\omega^*, \delta} \mathcal{S}_\omega (
T_n^j P_n^j \varphi^j ) .
\end{align}
We now claim for sufficiently large $k$ and $n$, $u_n^{< k} + e^{it\Delta} w_n^k$ is an approximate solution to $u_n$ in the sense of the stability theory. Then we have the finiteness of the $L_t^\frac{2(d+2)}d W_x^{s_p, \frac{2(d+2)}d}$ norm of $u_n$, which contradicts with \eqref{eq6.161v3}. To verify the claim, we only need to check
\begin{align}
& \limsup\limits_{n\to \infty} \left\|u_n^{< k} + e^{it\Delta} w_n^k \right\|_{L_t^\frac{2(d+2)}d W_x^{s_p, \frac{2(d+2)}d}} \lesssim_{m_\omega^*, \delta} 1, \text{ uniformly in $k$}, \label{eq5.28v14}\\
& \limsup\limits_{n\to \infty} \left\|e_n^k \right\|_{L_t^\frac{2(d+2)}d W_x^{s_p, \frac{2(d+2)}d}} \to 0, \text{ as } k \to \infty, \label{eq5.29v14}
\end{align}
where
\begin{align}
e_n^k &  := \left(i\partial_t + \Delta\right) \left( u_n^{< k} + e^{it\Delta} w_n^k\right) \nonumber \\
& \quad - \left| u_n^{< k}
+ e^{it\Delta} w_n^k \right|^\frac4d \left( u_n^{< k} + e^{it\Delta} w_n^k \right) + \left| u_n^{< k} + e^{it\Delta} w_n^k\right|^{p-1} \left( u_n^{< k} + e^{it\Delta} w_n^k\right) \nonumber \\
& = \sum\limits_{ j = 1}^k   \left|u_n^j\right|^\frac4d u_n^j  - \left| \sum\limits_{j = 1}^k u_n^j \right|^\frac4d \sum\limits_{j = 1}^k u_n^j  - \sum\limits_{j = 1}^k \left|u_n^j\right|^{p-1} u_n^j +  \left| \sum\limits_{j=1}^k u_n^j \right|^{p-1} \sum\limits_{j =1 }^k u_n^j \label{eq5.23v34}\\
&\quad  + \left|u_n^{< k} \right|^\frac4d u_n^{< k} - \left| u_n^{< k} + e^{it\Delta} w_n^k \right|^\frac4d \left( u_n^{< k} + e^{it\Delta} w_n^k\right) \nonumber \\
& \quad - \left|u_n^{< k}\right|^{p-1} u^{< k}_n + \left| u_n^{< k} + e^{it\Delta} w_n^k\right|^{p-1} \left( u_n^{< k} + e^{it\Delta} w_n^k\right).\label{eq5.24v34}
\end{align}
The verification of \eqref{eq5.28v14} relies on the following lemma, which can be proved by using Theorem \ref{th6.1v26} and the orthogonality relation \eqref{eq6.48v3}. We refer to \cite{CGZ} and will not prove it.
\begin{lemma}[Decoupling of the nonlinear profiles]\label{le5.3v15}
For $j\ne l$,
\begin{align*}
\left\|\langle \nabla \rangle^{s_p}  u_n^j \cdot \langle \nabla \rangle^{s_p} u_n^l \right\|_{L_{t,x}^\frac{d+2}d}
+ \left\|   u_n^j \cdot \langle \nabla \rangle^{s_p} u_n^l \right\|_{L_{t,x}^\frac{2(d+2)(p-1)}{d(p-1)+4}}
 \to 0, \text{ as } n\to \infty.
\end{align*}
\end{lemma}
We now verify \eqref{eq5.28v14}. By \eqref{eq5.28v15} and Lemma \ref{le5.3v15}, we have for $k$ large enough,
\begin{align*}
& \ \left\| \langle \nabla \rangle^{s_p}  \left( \sum\limits_{j=1}^k u_n^j \right)  \right \|_{L_{t,x}^\frac{2(d+2)}d}^\frac{2(d+2)}d\\
& \lesssim \left( \sum\limits_{j=1}^k  \left\| \langle \nabla \rangle^{s_p} u_n^j \right\|_{L_{t,x}^\frac{2(d+2)}d}^2 + \sum\limits_{j \ne l} \left\| \langle \nabla \rangle^{s_p}  u_n^j \cdot \langle \nabla \rangle^{s_p}  u_n^{l} \right\|_{L_{t,x}^\frac{d+2}d} \right)^\frac{d+2}d\\
& \lesssim \left( \sum\limits_{j = 1}^k \mathcal{S}_\omega\left( T_n^j P_n^j \varphi^j\right)+ o_k(1) \right)^\frac{d+2}d.
\end{align*}
 The decoupling of $\mathcal{S}_\omega$ implies
\begin{align*}
\sum\limits_{j= 1}^k \mathcal{S}_\omega( T_n^j P_n^j \varphi^j) \le m_\omega^*,
\end{align*}
together with \eqref{eq6.50v3}, we obtain \eqref{eq5.28v14}. It remains to check \eqref{eq5.29v14}, first we consider \eqref{eq5.23v34}.
By the fractional chain rule, we have
\begin{align*}
& \left\| \sum\limits_{ j = 1}^k   \left|u_n^j\right|^\frac4d u_n^j  - \left| \sum\limits_{j = 1}^k u_n^j \right|^\frac4d \sum\limits_{j = 1}^k u_n^j \right\|_{L_{t }^\frac{2(d+2)}d W_x^{s_p, \frac{2(d+2)}d }} \\
 & \quad + \left\|   \sum\limits_{j = 1}^k \left|u_n^j\right|^{p-1} u_n^j -  \left| \sum\limits_{j=1}^k u_n^j \right|^{p-1} \sum\limits_{j =1 }^k u_n^j \right\|_{L_{t }^\frac{2(d+2)}d W_x^{s_p, \frac{2(d+2)}d }}\\
 \lesssim &  \sum\limits_{ j \ne l } \left( \left\| \left| u_n^j\right|^\frac4d   \cdot  \langle \nabla \rangle^{s_p} u_n^{l} \right\|_{L_{t,x}^\frac{2(d+2)}{d+4}} +  \left\| \left|u_n^{l} \right| \left|u_n^j \right|^{\frac4d - 1}  \left| \langle \nabla \rangle^{s_p} u_n^j\right| \right\|_{L_{t,x}^\frac{2(d+2)}{d+4}} \right)  \\
&\quad  +  \left(\left\|  \left| u_n^j\right|^{p-1}   \cdot  \langle \nabla \rangle^{s_p} u_n^{l} \right\|_{L_{t,x}^\frac{2(d+2)}{d+4}} +  \left\| \left|u_n^{l} \right| \left|u_n^j \right|^{p-2}  \left| \langle \nabla \rangle^{s_p} u_n^j \right| \right\|_{L_{t,x}^\frac{2(d+2)}{d+4}}
 \right)\\
 \lesssim & \sum\limits_{j \ne l} \left\|u_n^j \cdot \langle \nabla \rangle^{s_p} u_n^{l} \right\|_{L_{t,x}^\frac{d+2}d} \left( \left\|u_n^j \right\|_{L_{t,x}^\frac{2(d+2)}d}^{\frac4d - 1}  + \left\|u_n^{l} \right\|_{L_{t,x}^\frac{2(d+2)}d}^{\frac4d - 1}\right)\\
 & \quad  +   \sum\limits_{j \ne l} \left\|u_n^j \cdot \langle \nabla \rangle^{s_p} u_n^{l} \right\|_{L_{t,x}^\frac{2(d+2)(p-1)}{d(p-1) + 4}} \left( \left\|u_n^j \right\|_{L_{t,x}^\frac{(d+2)(p-1)}2}^{p - 2} +  \left\|u_n^{l} \right\|_{L_{t,x}^\frac{(d+2)(p-1)}2}^{p - 2} \right)\\
  \lesssim&_{k, m_\omega^*, \delta} o(1), \text{ as } n \to \infty.
\end{align*}
We now consider \eqref{eq5.24v34}, by the fractional chain rule, the H\"older inequality,
\begin{align*}
 &  \left\| \left|u_n^{< k} \right|^\frac4d u_n^{< k} - \left| u_n^{< k} + e^{it\Delta} w_n^k \right|^\frac4d \left( u_n^{< k} + e^{it\Delta} w_n^k\right) \right\|_{L_{t }^\frac{2(d+2)}d W_x^{s_p, \frac{2(d+2)}d }}\\
 & +   \left\|   \left|u_n^{< k}\right|^{p-1} u^{< k}_n - \left| u_n^{< k} + e^{it\Delta} w_n^k\right|^{p-1} \left( u_n^{< k} + e^{it\Delta} w_n^k\right) \right\|_{L_{t }^\frac{2(d+2)}d W_x^{s_p, \frac{2(d+2)}d }}\\
\lesssim & \left\|\langle \nabla \rangle^{s_p} e^{it\Delta} w_n^k \right\|_{L_{t,x}^\frac{2(d+2)}d} \left\|e^{it\Delta} w_n^k \right\|_{L_{t,x}^\frac{(d+2)(p-1)}2}^{p-1}
+ \left\| \langle \nabla \rangle^{s_p} u_n^{< k} \right\|_{L_{t,x}^\frac{2(d+2)}d} \left\|e^{it\Delta} w_n^k \right\|_{L_{t,x}^\frac{(d+2)(p-1)}2}^{p-1} \\
& + \left\| \langle \nabla\rangle^{s_p} u_n^{<k} \right\|_{L_{t,x}^\frac{2(d+2)}d} \left\|e^{it\Delta} w_n^k \right\|_{L_{t,x}^\frac{(d+2)(p-1)}2} \left\|u_n^{<k} \right\|_{L_{t,x}^\frac{(d+2)(p-1)}2}^{p-2}\\
  & + \left\| u_n^{<k}\right\|_{L_{t,x}^\frac{(d+2)(p-1)}2}^{p-1} \left\|  \langle \nabla \rangle^{s_p} e^{it\Delta} w_n^k \right\|_{L_{t,x}^\frac{2(d+2)}{d}}\\
 & +  \left\|\langle \nabla \rangle^{s_p} e^{it\Delta} w_n^k \right\|_{L_{t,x}^\frac{2(d+2)}d} \left\|e^{it\Delta} w_n^k \right\|_{L_{t,x}^\frac{2(d+2) }d}^{\frac4d}
+ \left\| \langle \nabla \rangle^{s_p} u_n^{< k} \right\|_{L_{t,x}^\frac{2(d+2)}d} \left\|e^{it\Delta} w_n^k \right\|_{L_{t,x}^\frac{2(d+2)}d}^{\frac4d} \\
& + \left\| \langle \nabla\rangle^{s_p} u_n^{<k} \right\|_{L_{t,x}^\frac{2(d+2)}d} \left\|e^{it\Delta} w_n^k \right\|_{L_{t,x}^\frac{2(d+2)}d} \left\|u_n^{<k} \right\|_{L_{t,x}^\frac{2(d+2)}d}^{\frac4d-1} \\
& + \left\| u_n^{<k} \right\|_{L_{t,x}^\frac{2(d+2)}d}^{\frac4d} \left\| \langle \nabla \rangle^{s_p} e^{it\Delta} w_n^k \right\|_{L_{t,x}^\frac{2(d+2)}{d}}.
\end{align*}
Using \eqref{eq6.50v3} and \eqref{eq5.28v14}, we obtain the desired estimate.
\end{proof}
We now show the trajectory of the critical element is precompact in $H^1(\mathbb{R}^d)$ modulo spatial translations.
\begin{proposition}[Compactness of the critical element]\label{pr5.9}
Let $u_c$ be the critical element in Proposition \ref{pr5.8}, then there exists spatial translation parameter $x(t): \mathbb{R} \to \mathbb{R}^d$ such that $\left\{ u_c\left(t, x + x(t) \right): t\in \mathbb{R}\right\}$ is precompact in $H^1(\mathbb{R}^d)$.
\end{proposition}
\begin{proof}
For $\left\{t_n\right\} \subseteq \mathbb{R}$. If $t_n \to \infty$, applying the argument as deriving \eqref{eq6.183v3} to $u_c(t + t_n)$, there exist $t_n'\in \mathbb{R}$, $x_n' \in \mathbb{R}^d$, and $\phi\in H^1(\mathbb{R}^d)$ such that
\begin{equation}\label{eq:5.4}
u_c(t_n, x ) - e^{-it_n' \Delta} \phi(x - x_{n}')  \to 0 \ \text{  in }  H^1(\mathbb{R}^d), \text{ as } n\to \infty.
\end{equation}
{\rm (i)} If $t_n' \to - \infty$, then we have
\begin{align*}
\left\|\langle \nabla \rangle^{s_p} e^{it\Delta} u_c(t_n)\right\|_{L_{t,x}^\frac{2(d+2)}d([0,\infty)\times \mathbb{R}^d)}
= \left\|\langle \nabla \rangle^{s_p} e^{it\Delta} \phi\right\|_{L_{t,x}^\frac{2(d+2)}d([-t_n',\infty) \times \mathbb{R}^d)} + o_n(1)
\to 0,
\end{align*}
as $n\to \infty$.
Hence, we can solve \eqref{eq1.1} for $t > t_n$ globally by iteration with small Strichartz norm when $n$ large enough, which
contradicts
\begin{equation*}
\left\|\langle \nabla \rangle^{s_p} u_c\right\|_{L_{t,x}^\frac{2(d+2)}d ([0,\infty)\times \mathbb{R}^d)} = \infty.
\end{equation*}
{\rm (ii)} If $t_n' \to \infty$, then we have
\begin{align*}
\left\|\langle \nabla \rangle^{s_p}  e^{it\Delta} u_c(t_n)\right\|_{L_{t,x}^\frac{2(d+2)}d((-\infty,0]\times \mathbb{R}^d)}
= \left\|\langle \nabla\rangle^{s_p} e^{it\Delta} \phi\right\|_{L_{t,x}^\frac{2(d+2)}d((-\infty,-t_n'] \times \mathbb{R}^d)} + o_n(1)
\to 0,
\end{align*}
as $n\to \infty$.
 Hence, $u_c$ can solve \eqref{eq1.1} for $t < t_n$ when $n$ large enough with diminishing Strichartz norm, which contradicts
\begin{equation*}
\left\| \langle \nabla \rangle^{s_p} u_c\right\|_{L_{t,x}^\frac{2(d+2)}d  ((-\infty,0])\times \mathbb{R}^d)} = \infty.
\end{equation*}
Thus $t_n'$ is bounded, which implies that $t_n'$ is precompact, so $u_c(t_n, x + x_n')$ is precompact in $H^1(\mathbb{R}^d)$ by \eqref{eq:5.4}.

Similar argument makes sense when $t_n \to -\infty$ as above, so we will omit the proof.

If $t_n \to t^* \in \mathbb{R}, \text{ as } n\to \infty$, then we see by the continuity of $u_c(t, x  )$ in $t$ that
\begin{equation*}
u_c(t_n, x  ) \to u_c(t^*, x) \ \text{ in }  H^1(\mathbb{R}^d), \text{  as  } n \to \infty.
\end{equation*}
\end{proof}
As a consequence, we have
\begin{corollary}\label{cor5.3}
Let $u_c$ be the critical element in Proposition \ref{pr5.8}, then $\forall\, \varepsilon > 0$, there exists $R_0(\varepsilon) > 0$ and $x(t):\mathbb{R} \to \mathbb{R}^d$ such that $\forall\, t \in \mathbb{R}$,
\begin{equation}\label{eq5.28v35}
\int_{|x-x(t)|\ge R_0(\epsilon)}|\nabla u_c(t)|^2 + |u_c(t)|^2 + |u_c(t)|^{\frac{2(d+2)}d  } + |u_c(t)|^{p +1} \, \mathrm d x \le \varepsilon \mathcal{E}(u_c).
\end{equation}
Moreover, the momentum of $u_c$ is zero, and the spatial translation parameter $x(t)$ satisfies
\begin{align}\label{eq5.30v61}
\frac{|x(t)|}{t} \to 0, \quad \text{ as } t \to \pm\infty.
\end{align}
\end{corollary}
\begin{proof}
We only need to show $\mathcal{P}(u_c) = 0$. If $\mathcal{P}(u_c) \ne 0$, choosing $\xi_0 \in \mathbb R^d $ such that
\[-\frac12\mathcal{K}(u_c) \le  \left|\xi_0\right|^2\mathcal{M}(u_c) + 2\xi_0\cdot\mathcal{P}(u_c) < 0,\]
and for $\tilde u_c(t,x) = e^{ix\xi_0}e^{-it|\xi_0|^2}u_c(t,x-2\xi_0 t)$, we have
\begin{align*}
&\mathcal{S}_\omega(\tilde u_c) = \mathcal{E}(\tilde u_c) +\frac{\omega}{2}\mathcal{M}(\tilde u_c) < \mathcal{E}(u_c) + \frac{\omega}{2}\mathcal{M}( u_c) = \mathcal{S}_\omega(u_c),\\
&\mathcal{K}(\tilde u_c) = \mathcal{K}(u_c) + |\xi_0|^2\mathcal{M}(u_c) + 2\xi_0\cdot\mathcal{P}(u_c) \ge \frac12 \mathcal{K}(u_c)\ge 0,\\
& \left\|\langle \nabla \rangle^{s_p} \tilde u_c \right\|_{ L_{t,x}^{\frac{2(d+2)}{d}} ((0,\infty)\times \mathbb{R}^d) }
=   \ \left\| \langle \nabla \rangle^{s_p} \tilde u_c\right\|_{ L_{t,x}^{\frac{2(d+2)}{d}} ((-\infty,0) \times \mathbb{R}^d) }   = \infty.
\end{align*}
This contradicts the definition of $u_c$. We can obtain \eqref{eq5.30v61} by using the truncated center of mass together with $\mathcal{P}(u_c) = 0$. We will not prove it, and refer to \cite{Duyckaerts-Holmer-Roudenko,KM2,KTPV} for similar argument.
\end{proof}

\section{Extinction of the critical element}\label{se5}
In this section, we prove the non-existence of the critical element by deriving a contradiction from Corollary \ref{cor5.3} and the virial identity.

For a real-valued function $\phi\in C^\infty(\mathbb{R}^d)$, we can define the localize virial quantity:
\begin{equation}\label{eq6.1v10}
V_R(t) = \int_{\mathbb{R}^d} R^2 \phi\left(\frac{|x|}R\right)  |u(t,x)|^2  \,\mathrm{d}x, \,\forall \, R > 0.
\end{equation}
Then, for $u \in C(I; H^1(\mathbb{R}^d))$, we have
\begin{align}
V_R'(t)  = &    2R  \Im \int_{\mathbb{R}^d} \phi'\left(\frac{|x|}R \right) \frac{x}{|x|}  \cdot \nabla u(t,x) \,\overline{u(t,x)} \,\mathrm{d}x, \label{eq6.2v10}\\
V_R''(t)  = &   4\Re \int_{\mathbb{R}^d} \partial_j\partial_k \phi_R(x)
\overline{\partial_j u(t,x)} \partial_k u(t,x) \,\mathrm{d}x
- \int_{\mathbb{R}^d} \Delta^2 \phi_R(x) |u(t,x)|^2 \,\mathrm{d}x  \nonumber \\
   &   + \frac{4}{d+2 } \int_{\mathbb{R}^d} \Delta \phi_R(x) |u(t,x)|^{\frac{2(d+2)} d} \,\mathrm{d}x
- \frac{2(p -1)}{p +1} \int_{\mathbb{R}^d} \Delta \phi_R(x) |u(t,x)|^{p +1} \,\mathrm{d}x, \label{eq6.3v10}
\end{align}
where $\phi_R(x) = R^2 \phi\left(\frac{|x|}R\right)$.
\begin{theorem}[Nonexistence of the critical element] \label{th6.2v36}
The critical element $u_c$ in Proposition \ref{pr5.8} does not exist.
\end{theorem}
\begin{proof}
Let weighted function $\phi \in C_0^\infty(\mathbb{R}^d)$ in \eqref{eq6.1v10} be radial with
\begin{align}\label{eq6.12v35}
\phi(x)=
\begin{cases}
|x|^2, & |x| \le 1,\\
0, & |x| \ge 2,
\end{cases}
\end{align}
then by \eqref{eq6.2v10}, the H\"older inequality and Lemma \ref{le2.8},
\begin{align}\label{eq:6.6}
|V_R^{'}(t)| \le    CR\int_{|x| \le 2R} \left|\nabla u_c(t,x)\right| \left|u_c(t,x)\right|\, \mathrm{d} x \le    CR \left\|u_c\right\|_{L^2} \left\|\nabla u_c\right\|_{L^2} \lesssim R.
\end{align}
On the other hand, by \eqref{eq6.3v10} and direct calculation, we have
\begin{align}
V_R^{''}(t)  =   8\mathcal{K}(u_c) + A_R(u_c(t)),\label{eq:6.7}
 \end{align}
where
\begin{align*}
& A_R(u_c(t)) \\
= &  \  4\sum_{j=1}^d\int_{\mathbb R^d}\Big((\partial_j^2\phi)\Big(\frac{x}{R}\Big)-2\Big)|\partial_ju_c|^2\, \mathrm{d}x \\
 &  + 4\Re \sum_{1\le j\ne l\le d }\int_{R \le |x| \le 2R}(\partial_j\partial_l \phi)\Big(\frac{x}{R}\Big)\partial_j\bar u_c \partial_l  u_c \, \mathrm{d} x - \frac{1}{R^2}\int_{\mathbb R^d}(\Delta^2\phi)\Big(\frac{x}{R}\Big)|u_c|^2 \, \mathrm{d}x \\
 & + \frac{4}{d + 2 }\int_{\mathbb R^d}\Big((\Delta\phi)\Big(\frac{x}{R}\Big) -2d\Big)|u_c(t,x)|^{\frac{2(d+2)}d} \ \mathrm d x \\
  & - \frac{2(p -1)}{p +1}\int_{\mathbb R^d}\Big((\Delta\phi)\Big(\frac{x}{R}\Big) -2d\Big)|u_c(t,x)|^{p +1} \, \mathrm{d} x.
\end{align*}
By the choice of the weighted function $\phi$ in \eqref{eq6.12v35}, we obtain the bound
\begin{align}\label{eq:6.8}
|A_R(u_c(t))| \le C\int_{|x| \ge R} |\nabla u_c|^2 + \frac{1}{R^2}|u_c|^2 + |u_c|^{ \frac{2(d+2)}d } + |u_c|^{p +1} \, \mathrm{d} x.
\end{align}
We want to examine $V_R(t)$, for $R$ chosen suitably large, over a suitably chosen time interval $[t_0, t_1]$, where $1<< t_0 << t_1 <\infty$.

By \eqref{eq5.30v61}, we have $|x(t)| \le \eta t$, for all $t \ge t_0$, with $\eta > 0$ to be selected later. Thus, by taking $R = R_0 + \eta t_1$, where $R_0 = R(\epsilon)$ is taken in \eqref{eq5.28v35}, then \eqref{eq:6.7} combined with the bounds \eqref{eq:6.8} and \eqref{eq5.28v35} will imply that, for all $t_0 \le t \le t_1$,
\begin{align}\label{eq:6.11}
\left|V_R^{''}(t)\right| \ge 8\mathcal{K}(u_c(t)) -  C \varepsilon.
\end{align}
Integrating \eqref{eq:6.11} over $[t_0, t_1]$, we obtain
\begin{align}\label{eq:6.12}
\left|V_R^{'}(t_1) -V_R^{'}(t_0) \right| \ge \left(8\mathcal{K}(u_c(t)) - C  \varepsilon\right) \left(t_1-t_0 \right).
\end{align}
On the other hand, by \eqref{eq:6.6}, we have
\begin{align}\label{eq:6.13}
\left|V_R^{'}(t)\right| \lesssim R, \ \forall\, t_0 \le t \le t_1.
\end{align}
Combining \eqref{eq:6.12}, \eqref{eq:6.13}, we obtain
\[\left(8\mathcal{K}(u_c(t)) - C \varepsilon \right)\left(t_1-t_0 \right) \lesssim R.\]
By Lemma \ref{le2.6} and Proposition \ref{le2.9}, we have $\mathcal{K}(u_c) \gtrsim \mathcal{E}(u_c)> 0$, we can take $\varepsilon > 0$ small enough such that $8\mathcal{K}(u_c(t)) - C\varepsilon \gtrsim \mathcal{E}(u_c) $, we then obtain
\[\mathcal{E}(u_c)(t_1-t_0) \lesssim R \lesssim R_0 + \eta t_1.\]
We now take $\eta > 0$ small enough and then send $t_1 \to \infty$ to obtain a contradiction unless $\mathcal{E}(u_c) = 0$, which implies $u_c  \equiv 0$.
\end{proof}

\section{Sketch of the proof of Theorem \ref{th1.3v16} } \label{se8v19}
In this section, we give a sketch of the proof of Theorem \ref{th1.3v16}. Most of the argument are similar to the proof of Theorem \ref{th1.4}. The equation \eqref{eq1} has the conservation quantities:
\begin{align*}
\mathcal{M}(u)    = \int_{\mathbb{R}^d} |u|^2 \,\mathrm{d}x, \
\mathcal{E}(u)     = \int_{\mathbb{R}^d} \frac12 |\nabla u|^2 + \frac1{p+1} |u|^{p+1} - \frac{d}{2(d+2)} |u|^\frac{2(d+2)}d \,\mathrm{d}x.
\end{align*}
We will rely on the sharp Galiardo-Nirenberg inequality:
\begin{theorem}[Sharp Galiardo-Nirenberg inequality, \cite{W}]
Let $Q$ be the ground state solution of \eqref{eq1.10v30}, we have
\begin{align*}
\int_{\mathbb{R}^d} |f(x)|^\frac{2(d+2)}d \,\mathrm{d}x \le \frac{d+2}d \left( \frac{ \|f\|_{L^2} }{ \|Q\|_{L^2}} \right)^\frac4d \int_{\mathbb{R}^d} |\nabla f(x)|^2\,\mathrm{d}x.
\end{align*}
\end{theorem}
By the sharp Galiardo-Nirenberg inequality together with the the assumption $\|u_0\|_{L^2} < \|Q\|_{L^2}$ in Theorem \ref{th1.3v16} and mass-conservation, we have
\begin{align}\label{eq8.2v36}
\mathcal{E}(u)   \ge  \frac12 \left( 1-  \left( \frac{ \|u\|_{L^2}}{ \|Q\|_{L^2}} \right)^\frac4d \right) \|\nabla u \|_{L^2}^2 + \frac1{p+1} \|u \|_{L^{p+1}}^{p+1} \gtrsim \|\nabla u \|_{L^2}^2 + \|u \|_{L^{p+1}}^{p+1}.
\end{align}
Therefore, we get the energy can control the kinetic energy, similar as the estimate in Lemma \ref{le2.8}. The wellposedness and stability theory in Section \ref{se4v24} can be modified easily to adapt to \eqref{eq1}, while the linear profile decomposition in Section \ref{se5v24} can be applied directly without any change. The induction on energy argument, especially $m_\omega^*$ in \eqref{eq5.2} in the beginning of Section \ref{se6v24} should be replaced by
\begin{align*}
m^* = \sup \left\{ m \in \left( 0, \mathcal{M}(Q) \right) : \Lambda(m) < \infty\right\},
\end{align*}
with
\begin{align*}
\Lambda(m)  = \sup \left\| \langle \nabla \rangle^{s_p} u\right\|_{L_{t,x}^\frac{2(d+2)}d},
\end{align*}
where the supremum is taken over the solution $u$ to \eqref{eq1} with $u_0 \in H^1$ and $\mathcal{M}(u_0) \le m < \mathcal{M}(Q)$. The approximate solution in Theorem \ref{th6.1v26} should be a solution of the mass-critical focusing nonlinear Schr\"odinger equation instead of the defocusing equation as for \eqref{eq1.1}, and we need to use the scattering theorem of the mass-critical focusing nonlinear Schr\"odinger equation when the mass is less than the mass of the ground state solution \cite{D4}. The remaining part in Section \ref{se6v24} follows with proper modification. In Section \ref{se5}, for the virial quantity
\begin{equation*}
V(t) = \int_{\mathbb{R}^d} |x|^2  |u(t,x)|^2  \,\mathrm{d}x,
\end{equation*}
as in \eqref{eq6.1v10} when $R = 1$, we have
\begin{align*}
V''(t)  = & 8 \int_{\mathbb{R}^d} |\nabla u |^2 +  \frac{d(p-1)}{2(p+1)} |u |^{p+1} - \frac{d}{d+2} |u |^\frac{2(d+2)}d \,\mathrm{d}x.
\end{align*}
By the sharp Gagliardo-Nirenberg inequality as \eqref{eq8.2v36}, we have
\begin{align*}
V''(t) \gtrsim \|\nabla u \|_{L^2}^2 + \|u \|_{L^{p+1}}^{p+1} \gtrsim \mathcal{E}(u),
\end{align*}
thus we need to modify the lower bound in \eqref{eq:6.11} with some constant $c = c(d,p)$ times $\mathcal{E}(u)$ instead of $\mathcal{K}(u)$.

\textbf{Acknowledgements}
{Xing Cheng has been partially supported by the NSF grant of China (No.11526072). Xing Cheng is grateful to S. Le Coz for calling attention to their work \cite{Le Coz-Tsai}, and also thank A. Stefanov for sending his recent work \cite{St} to the author.}

\end{document}